\documentclass[12pt,leqno]{article}
\usepackage{fullpage}
\usepackage{amssymb}
\usepackage{amsmath}
\usepackage{amsthm}
\newtheorem{theorem}{Theorem}

\newtheorem{lemma}{Lemma}
\newtheorem{remark}{Remark}
\begin{document}

\title{\Large On the sumsets of exceptional units in $\mathbb{Z}_n$}
\author{\large Quan-Hui Yang$^{1}$\footnote{
Emails:~yangquanhui01@163.com,~zhaoqingqing116@163.com.}
~~Qing-Qing Zhao$^{2}$}
\date{} \maketitle
 \vskip -3cm
\begin{center}
\vskip -1cm { \small
1. School of Mathematics and Statistics, Nanjing University of Information \\
Science and Technology, Nanjing 210044, China }
 \end{center}

 \begin{center}
{ \small 2. Wentian College, Hohai University,  Maanshan 243031,
China }
 \end{center}

\begin{abstract} Let $R$ be a commutative ring with $1\in R$ and
$R^{\ast}$ be the multiplicative group of its units. In 1969,
Nagell introduced the exceptional unit $u$ if both $u$ and $1-u$
belong to $R^{\ast}$. Let $\mathbb{Z}_n$ be the ring of residue
classes modulo $n$. In this paper, given an integer $k\ge 2$, we
obtain an exact formula for the number of ways to represent each
element of $ \mathbb{Z}_n$ as the sum of $k$ exceptional units.
This generalizes a recent result of J. W. Sander for the case
$k=2$.



\end{abstract}

\section{Introduction}

\quad \quad Let $R$ be a commutative ring with $1\in R$ and
$R^{\ast}$ be the multiplicative group of its units. In 1969,
Nagell introduced the concept of exceptional units. A unit $u\in
R$ is called {\em exceptional} if $1-u\in R^{\ast}$. Since the
solution of many Diophantine equations can be reduced to the
solution of $ax+by=1$ with $x,y$ are units in some ring, this
means to find exceptional units in the case $a=b=1$. Hence the
exceptional units are very important for studying Diophantine
equations, such as certain cubic Diophantine equations
\cite{Nagell}, Thue equations \cite{Tzanakis}, Thue-Mahler
equations \cite{Tzanakis92}, discriminant form equations
\cite{Smart} and so on (one can refer to \cite{Niklasch},
\cite{Niklasch98}).

In 1977, by using exceptional units, Lenstra \cite{Lenstra}
introduced a method to find Euclidean number fields. After that,
many new Euclidean number fields were found (See \cite{DeaDu} and
\cite{Leutbecher}). Beyond these, exceptional units also have
connections with cyclic resultants \cite{Stewart,Stewart13} and
Lehmer's conjecture related to Mahler's measure
\cite{Silverman,Silverman96}.

Let $R^{\ast\ast}$ be the set of all exceptional units in $R$. In
this paper, we consider the ring $R=\mathbb{Z}_n$ of residue
classes mod $n$. By definition, we have
$$\mathbb{Z}_n^{\ast}=\{a\in \mathbb{Z}_n:\gcd(a,n)=1\},$$
$$\mathbb{Z}_n^{\ast\ast}=\{a\in \mathbb{Z}_n:\gcd(a,n)=1~\text{and}~\gcd(a-1,n)=1\}.$$

For a prime $p$, we write $p^a\|n$ if $p^a|n$ and $p^{a+1}\nmid
n$. Throughout this paper, we write $e(x)=e^{2\pi ix}$ and $p$
always denotes a prime. We assume that $0\le a\le n-1$ if $a\in
\mathbb{Z}_n$. In 2010, Harrington and Jones \cite{Harrington}
proved that
$$\sharp ~\mathbb{Z}_n^{\ast\ast}=n\prod_{p|n}\left(1-\frac 2p\right),$$
which also follows immediately from results of Deaconescu
\cite{Deaconescu} or Sander \cite{Sander}.

In \cite[Theorem 1.1]{Sander}, the number of representations of an
element $c\in \mathbb{Z}_n$ as the sum of two units was determined
to be
$$\varphi^{\ast}(n,c):=n\prod_{p|n,p|c}\left(1-\frac
1p\right)\prod_{p|n,p\nmid c}\left(1-\frac 2p\right).$$

For any integers $k,n,c$ with $k,n\ge 2$, let
$$\varphi_k(n,c)=\sharp\{(x_1,x_2,\ldots,x_k)\in (\mathbb{Z}_n^{\ast\ast})^k:~x_1+x_2+\cdots+x_k\equiv c \pmod n\}.$$
Recently, Sander \cite{Sander15} gave an exact formula for
$\varphi_2(n,c)$. In order to present Sander's result, we define a
function $\varphi^{\ast\ast}(n,c)$ below.

Set $\varphi^{\ast\ast}(1,c):=1$. For all integers $k\ge 1$, let
$\varphi^{\ast\ast}(2^k,c):=0$ and
$\varphi^{\ast\ast}(3^k,c):=3^{k-1}$ if $c\equiv 1~(\text{mod}~3)$
and $0$ otherwise. For all primes $p\ge 5$, let
$$\varphi^{\ast\ast}(p^k,c):=\left\{\begin{array}{l l}
p^{k-1}(p-2)& \text{if}~~c\equiv 1 ~(\rm{mod}~p), \\
p^{k-1}(p-3)& \text{if}~~c\equiv 0 ~({\rm{mod}}~p)~\text{or}~c\equiv 2 ~({\rm{mod}}~p),\\
p^{k-1}(p-4)& \text{otherwise}.
\end{array}\right.$$
Define $\varphi^{\ast\ast}(n,c)$ by multiplicative continuation
with respect to $n$.

\noindent{\bf  Sander's Theorem.} (See \cite[Theorem
1.1]{Sander15}) {\em~ Given integers $n>0$ and $c$, we have
$\varphi_2(n,c)=\varphi^{\ast\ast}(n,c)$.}

In this paper, we generalize Sander's theorem in the following.

\begin{theorem}\label{mainthm} For any integers $k,n,c$ with
$k,n\ge 2$, we have
$$\varphi_k(n,c)=(-1)^k\prod_{p^{\alpha}\|n}p^{\alpha k-\alpha-k}\left(p\sum_{\substack{j=0\\j\equiv c
~(\text{\rm mod}~p)}}^k {k \choose j}+(2-p)^k-2^k\right).$$
\end{theorem}

\begin{remark} Taking $k=2$, we can obtain Sander's result from
Theorem \ref{mainthm} by simple calculation.
\end{remark}

\section{Proofs}

\begin{lemma}\label{lem1} For any integers $k,n,c$ with $k,n\ge 2$, $\varphi_k(n,c)$ is multiplicative
with respect to $n$.
\end{lemma}

The proof of Lemma \ref{lem1} is similar to that of Theorem 1.1 in
\cite{Sander} and we leave it to the readers.

\begin{lemma}\label{lem2} For any prime power $p^{\alpha}$,
$\varphi_k(p^{\alpha},c)=p^{(k-1)(\alpha-1)}\varphi_k(p,c)$.
\end{lemma}
\begin{proof}[Proof.]
Let
\begin{eqnarray*}S&=&\{(x_1,x_2,\ldots,x_{k-1}):~x_1,x_2,\ldots,x_{k-1}\in
\mathbb{Z}_p^{\ast\ast}~\text{and there exists}\\
&& x_k\in \mathbb{Z}_p^{\ast\ast} ~\text{such
that}~x_1+x_2+\cdots+x_{k-1}+x_k\equiv c \pmod p
\}.\end{eqnarray*} Then $\varphi_k(p,c)=|S|$. Let
\begin{eqnarray*}T&=&\{(x_1+i_1p,x_2+i_2p,\ldots,x_{k-1}+i_{k-1}p):~(x_1,x_2,\ldots,x_{k-1})\in
S,\\
&&~0\le i_1,i_2,\ldots,i_{k-1}\le p^{\alpha-1}-1
\}.\end{eqnarray*} Clearly,
$|T|=p^{(k-1)(\alpha-1)}|S|=p^{(k-1)(\alpha-1)}\varphi_k(p,c)$.

Next, it suffices to prove $\varphi_k(p^{\alpha},c)=|T|$.

For any $k$-tuple $(x_1',x_2',\ldots,x_k')$ with $x_i'\in
\mathbb{Z}_{p^{\alpha}}^{\ast\ast}(1\le i\le k)$ and
\begin{eqnarray}\label{eq1}x_1'+x_2'+\cdots+x_k'\equiv c \pmod {p^{\alpha}},\end{eqnarray}
 we have
$((x_1')_p,\ldots,(x_{k-1}')_p)\in S$, and so
$(x_1',\ldots,x_{k-1}')\in T$. Since $x_k'$ is unique if
$x_i'(1\le i\le k-1)$ are fixed in \eqref{eq1}, it follows that
$\varphi_k(p^{\alpha},c)\le |T|$.

On the other hand, for any $(x_1',x_2',\ldots,x_{k-1}')\in T$, we
have $$((x_1')_p,(x_2')_p,\ldots,(x_{k-1}')_p)\in S.$$ By the
definition of the set $S$, there exists a unique $x_k''\in
\mathbb{Z}_p^{\ast\ast}$ such that
\begin{eqnarray}\label{eq2}(x_1')_p+(x_2')_p+\cdots+(x_{k-1}')_p+x_k''\equiv c \pmod {p}.\end{eqnarray}
Since there exists a unique $x_k'\in \mathbb{Z}_{p^{\alpha}}$ such
that \eqref{eq1} holds, by \eqref{eq2}, we have $x_k'\equiv
x_k''\pmod p$ and so $x_k'\in \mathbb{Z}_{p^{\alpha}}^{\ast\ast}$.
Hence, for any $(x_1',x_2',\ldots,x_{k-1}')\in T$, there exists
$x_k'\in \mathbb{Z}_{p^{\alpha}}^{\ast\ast}$ such that \eqref{eq1}
holds, and so $\varphi_k(p^{\alpha},c)\ge |T|$.

Therefore,
$\varphi_k(p^{\alpha},c)=|T|=p^{(k-1)(\alpha-1)}\varphi_k(p,c)$.

\end{proof}

\begin{proof}[Proof of Theorem \ref{mainthm}.] We first calculate
$\varphi_k(p,c)$ for prime $p$ and integer $c$.
\begin{eqnarray*}\varphi_k(p,c)&=&\sharp\{(x_1,x_2,\ldots,x_k)\in (\mathbb{Z}_p^{\ast\ast})^k:~x_1+x_2+\cdots+x_k\equiv
c~(\text{mod}~p)\}\\
&=&\sum_{x_1=2}^{p-1}\sum_{x_2=2}^{p-1}\cdots\sum_{x_k=2}^{p-1}\frac
1p\cdot\sum_{t=0}^{p-1}e\left(\frac{(x_1+x_2+\cdots+x_k-c)t}{p}\right)\\
&=&\frac 1p
\sum_{t=0}^{p-1}\left(\sum_{x_1=2}^{p-1}e\left(\frac{x_1t}{p}\right)\right)\cdots\left(\sum_{x_k=2}^{p-1}
e\left(\frac{x_kt}{p}\right)\right)e\left(\frac{-ct}{p}\right)\\
&=&\frac
1p\left(\sum_{t=1}^{p-1}\left(-1-e\left(\frac{t}{p}\right)\right)^k
e\left(\frac{-ct}{p}\right)+(p-2)^k\right)\\
&=&\frac 1p\left((-1)^k\sum_{j=0}^k{k\choose
j}\sum_{t=1}^{p-1}e\left(\frac{t(j-c)}{p}\right)+(p-2)^k\right)\\
&=&\frac {(-1)^k}{p}\left(\sum_{\substack{j=0\\j\equiv
c~\text{mod}~p}}^k{k\choose
j}(p-1)-\left(\sum_{\substack{j=0\\j\not\equiv c~\text{mod}~p}}^k
{k \choose j}\right)+(2-p)^k\right)\\
&=&\frac {(-1)^k}{p}\left(p\sum_{\substack{j=0\\j\equiv c
~(\text{mod}~p)}}^k {k \choose j}+(2-p)^k-2^k\right).
\end{eqnarray*}
By Lemma \ref{lem1} and Lemma \ref{lem2}, we have
$$\varphi_k(n,c)=(-1)^k\prod_{p^{\alpha}\|n}p^{\alpha k-\alpha-k}\left(p\sum_{\substack{j=0\\j\equiv c
~(\text{mod}~p)}}^k {k \choose j}+(2-p)^k-2^k\right).$$

\end{proof}

\section{Acknowledgement} This work was supported by the National Natural Science Foundation
for Youth of China, Grant No. 11501299, the Natural Science
Foundation of Jiangsu Province, Grant Nos. BK20150889,~15KJB110014
and the Startup Foundation for Introducing Talent of NUIST, Grant
No. 2014r029.
%
%

\clearpage
\end{document}